\documentclass[12pt]{amsart}
\usepackage[utf8x]{inputenc}
\usepackage{amsmath}
\usepackage{mathrsfs}
\usepackage{amstext}
\usepackage{amsxtra}
\usepackage{amssymb}
\usepackage{amsgen}
\usepackage{amsfonts}
\usepackage{amsthm}
\usepackage{esint}
\usepackage{amscd}
\usepackage{latexsym}
\usepackage{mathabx}

\usepackage{fullpage}

\usepackage{cite}
\usepackage{setspace}
\usepackage{bbm}
\usepackage{cancel}
\usepackage{color}

\numberwithin{equation}{section}

\newtheorem{theorem}{Theorem}[section]

\newtheorem{proposition}{Proposition}[section]

\newtheorem{definition}{Definition}[section]

\theoremstyle{remark}
\newtheorem{remark}{Remark}[section]

\title{Existence and Uniqueness of Energy Solutions to the Stochastic Diffusive Surface Quasi-Geostrophic Equation with Additive Noise}
\author{Nathan Totz}

\begin{document}

\maketitle

\begin{abstract}
We continue our study of the dynamics of a nearly inviscid periodic surface quasi-geostrophic equation.  Here we consider a slightly diffusive stochastic SQG equation of the form
\begin{equation*}
\begin{cases}
d\theta_t + |D|^{2\delta}\theta_t\,dx + (u_t \cdot \nabla)\theta_t\,dx + |D|^{\delta}dW_t = 0 \\
u_t = \nabla^\perp|D|^{-1}\theta_t.
\end{cases}
\end{equation*}
Analogous to the previous work \cite{TotzMSQG}, we find that we can construct global energy solutions as introduced in \cite{gonjar2014} for any $\delta > 0$, so that any small amount of diffusion permits us to construct solutions.  We show moreover that pathwise uniqueness of these energy solutions holds in the presence of sufficiently high diffusion $\delta > \frac32$.
\end{abstract}

\section{Introduction}

We consider the existence and uniqueness of global solutions to the following stochastic SQG equation defined on $\mathbb{T}^2$:
\begin{equation}\label{SQGTheta}
\begin{cases}
d\theta_t + |D|^{2\delta}\theta_t\,dx + (u_t \cdot \nabla)\theta_t\,dx + |D|^{\delta}dW_t = 0 \\
u_t = \nabla^\perp|D|^{-1}\theta_t,
\end{cases}
\end{equation}
where $\nabla^\perp := (-\partial_y, \partial_x)$ and $|D|$ is the Fourier multiplier with symbol $|\xi|$, and $dW_t$ is space-time white noise.  This system is a generalization of the classical deterministic inviscid surface quasi-geostrophic equation:

\begin{equation}\label{InviscidSQGTheta}
\begin{cases}
\frac{\partial\theta}{\partial t} + (u \cdot \nabla)\theta = 0 \\
u = \nabla^\perp|D|^{-1}\theta.
\end{cases}
\end{equation}

Equation \eqref{InviscidSQGTheta} is a geophysical model in atmospheric sciences which has been systematically studied by Constantin, Majda, and Tabak \cite{comata2}, who were motivated to study \eqref{InviscidSQGTheta} because of its formal similarity to the 3D Euler equations; by Pierrehumbert, Held, and Swanson; by Held, Pierrehumbert, Garner and Swanson and others (see \cite{comata1, comata2, pierre, held} and references therein).
This equation has since attracted a lot of attention in the pure mathematics community. Many interesting results describing the behavior of \eqref{InviscidSQGTheta} have been obtained, see e.g. \cite{re, co, cofe, wu, caco}.  The question of global in time existence and uniqueness of strong solutions to \eqref{SQGTheta} or global regularity is still an outstanding open problem, just as for the 3D Euler equations.  We recall that a solution $\theta$ of \eqref{InviscidSQGTheta} conserves $\|\theta\|_{L^p}$ for all $1 \leq p \leq \infty$, as well as the $\dot{H}^{-\frac12}$-norm.

Resnick \cite{re} proved the existence of deterministic weak solutions to \eqref{InviscidSQGTheta} in the class $\theta \in L^2$, and Marchand \cite{Marchand08} extended to the classes $\theta \in L^p, H^{-\frac12}$.  These results crucially rely on the classical conserved quantities of the equation.  In \cite{TotzMSQG}, the authors followed a method of \cite{alcr} to construct solutions almost surely in the class $\theta \in H^{-3^-}$ of a related modified SQG equation

\begin{equation}\label{ModifiedSQGTheta}
\begin{cases}
\frac{\partial\theta}{\partial t} + (u \cdot \nabla)\theta = 0 \\
u = \nabla^\perp|D|^{-1 - \delta}\theta
\end{cases}
\end{equation}
almost surely with respect to an invariant Gaussian measure which serves as a substitute for the conserved quantities at low regularity.  The results in \cite{TotzMSQG} required $\delta > 0$.  This is because in the $\delta = 0$ case the nonlinearity of \eqref{ModifiedSQGTheta} fails to make sense even in the distributional sense, and the consequently the construction method used in \cite{alcr} fails.

In an attempt to understand the well-posedness of \eqref{InviscidSQGTheta}, the following diffusive modification of SQG has been proposed:

\begin{equation}\label{DiffusiveSQGTheta}
\begin{cases}
\frac{\partial\theta}{\partial t} + |D|^{2\delta}\theta + (u \cdot \nabla)\theta = 0 \\
u = \nabla^\perp|D|^{-1}\theta.
\end{cases}
\end{equation}

Equation \eqref{DiffusiveSQGTheta} inherits a notion of criticality associated to the $L^\infty$ norm of the solution, since this is the strongest of the norms that are \textit{a priori} conserved in the inviscid case \eqref{InviscidSQGTheta}.  We call the equation subcritical, critical, and supercritical when $\delta$ is greater, equal to, or less than $\frac12$, respectively.

The global well-posedness of $L^2$ solutions in the subcritical case $\delta > \frac12$ was established by Constantin-Wu \cite{cw}.  Global existence of solutions in the more delicate critical case $\delta = \frac12$ were independently shown by several authors \cite{knv}, \cite{cv}, \cite{convic} using a variety of novel techniques.  As with the inviscid equation, global well-posedness in the supercritical case $\delta < \frac12$ remains an open problem.  However, there are a number of well-posedness results both local and global for sufficiently small initial data, c.f. \cite{coco, ju, zelvic} and their references, as well as regularity results (see for example \cite{silvestre2010} and its references).

In this paper we consider a different method of randomizing \eqref{InviscidSQGTheta} formed by adding a stochastic forcing term along with a diffusion term, yielding our original equation \eqref{SQGTheta}.  The particular choice of diffusive term and noise term are balanced so that the Gaussian invariant measure agrees with the conserved $L^2$ norm of the solution.  This invariant Gaussian measure plays a key role in our analysis.

The fundamental obstacle in constructing solutions to nonlinear stochastic PDE is that their solutions must necessarily lie in extremely rough spaces of functions in which the nonlinearities of the equation cannot be directly defined in the sense of distributions.  There are by now a number of differing versions of probabilistic solutions that overcome this problem, including the \textit{paracontrolled solutions} introduced by Gubinelli-Imkeller-Perkowski \cite{gip}, as well as the \textit{regularity structures} developed by Hairer \cite{hairer}.  We choose here to study the so-called \textit{energy solutions} of \eqref{SQGTheta}, a forerunner of paracontrolled solutions introduced in \cite{gonjar2014}, used in \cite{gj} to study weak probabilistic solutions to the stochastic Burgers equation, and elaborated upon in the lecture notes \cite{gpbrazil}.  For the precise definition of energy solutions, see Definitions \ref{ControlledProcessDefinition} and \ref{EnergySolutionDefinition}.  Our result is contained in the following

%Parallel with DaPrato-DeBussche, why we don't follow them.

%Comparison to Past Results.

\begin{theorem}
Let $\delta > 0$ and $T > 0$ be given.  There is an $\epsilon_0 > 0$ depending on $\delta$ and $T$ so that for every $\epsilon \in (0, \epsilon_0)$ the following holds:
\begin{itemize}
\item{(Existence)  Denote $s = (-\epsilon) \wedge (2\delta - 1 - 3\delta\epsilon)$.  There exists a probability space $(\Omega, \mathbb{P}, \mathfrak{F})$ so that for every $\theta_0 \in \mathcal{F}L^{\infty, s}$, there exists almost surely an energy solution $\theta$ to \eqref{SQGTheta} with initial data $\theta_0$ with mean zero in the class $C([0, T] : \mathcal{F}L^{\infty, s})$.  The solutions are stationary with law given by a Gaussian measure with correlation $\frac12\|\theta\|_{L^2}^2$.}
\item{(Uniqueness)  If moreover $\delta > \frac32$, then energy solutions to \eqref{SQGTheta} are pathwise unique in $C([0, T] : \mathcal{F}L^{\infty, s})$, and consequently can be extended to a global solution in the class $C([0, \infty) : \mathcal{F}L^{\infty, s})$.}
\end{itemize}
\end{theorem}

\begin{remark}
The Fourier-Lebesgue space $\mathcal{F}L^{\infty, s}$ above is defined in the next section and corresponds at the level of scaling to the standard $L^2$-based Sobolev class $H^{-1^-}$ in the subcritical regimes $\delta > \frac12$, and $H^{(2\delta - 2)^-}$ in the critical and supercritical regimes $0 < \delta \leq \frac12$.  The regularity class does not increase with increasing diffusion in the subcritical regime due to the dominance of the linear terms of the equation and our choice of regularity of the stochastic forcing term.
\end{remark}

\begin{remark}
We observe that, just as in \cite{TotzMSQG}, the endpoint case $\delta = 0$ corresponding formally to a stochastic inviscid SQG cannot be treated with the method used here.
\end{remark}

\begin{remark}
The sense of pathwise uniqueness that we use here is made precise in Definition \ref{PathwiseUniquenessDefinition}.
\end{remark}

\textbf{Acknowledgments.}  The author would like to thank Andrea Nahmod for her many helpful discussions.  The work done in this paper was partially supported by A. Nahmod's grants NSF DMS-1463714 and NSF DMS-1800852.  The author would also like to thank Gigliola Staffilani and Natasa Pavlovic for their comments on earlier versions of this draft.

\subsection{Notation and Analytic Setting}

\subsubsection{Functional Analytic Setting}  Let $\mathbb{Z}_0^2 = \mathbb{Z}^2 \setminus \{0\}$, and let $(e_k)_{k \in \mathbb{Z}^2}$ be a complex orthonormal basis of $L^2$ with $e_0 = 1$ and $e_{-k}(x) = \overline{e}_k(x)$.  We take the Fourier transform of $f \in \mathbb{T}^2$ to be $$\hat{f}(k) = \frac{1}{4\pi^2}\int_{\mathbb{T}^2} f(x)e_{-k}(x) \, dx,$$ so that $$f(x) = \sum_{k \in \mathbb{Z}} \hat{f}(k)e_k(x).$$
Define the space of Schwarz functions $\mathscr{S}(\mathbb{T}^2)$ by the space of functions $f \in C(\mathbb{T}^2)$ for which the seminorms $\sup_{k \in \mathbb{Z}}| k^\alpha \hat{f}(k)|$ are bounded for any choice of $\alpha$. The space of tempered distributions $\mathscr{S}^\prime(\mathbb{T}^2)$ consists of the linear functionals $f : \mathscr{S} \to \mathbb{C}$ continuous under the $\text{weak}^*$ topology induced by the above seminorms on $\mathscr{S}$.  Accordingly, for any $f \in \mathscr{S}^\prime$ the Fourier transform $\hat{f}$ for $f$ is well-defined by duality, and moreover $|\hat{f}(k)|$ grows at most polynomially in $k$.  

We will only consider tempered distributions $f$ with mean zero $\hat{f}(0) = 0$.  For any $s \in \mathbb{R}$ and $1 \leq p \leq \infty$, the Fourier-Lebesgue spaces $\mathcal{F}L^{p, s}$ are defined by $$\mathcal{F}L^{p, s} = \left\{f \in \mathscr{S}^\prime : \|f\|_{\mathcal{F}^{p, s}}^p := \sum_{k \in \mathbb{Z}_0^2} |k|^{ps}|\hat{f}(k)|^p < \infty \right\}$$ with the usual modification in the case $p = \infty$.  We further denote by $\mathcal{F}L^{2, s} = H^s$, which is equivalent to the usual $L^2$ Sobolev spaces through Plancherel's identity.

We will often identify $f = \sum_k f_k e_k$ with an element $(f_k)_{k \in \mathbb{Z}^2}$ of the countably infinite space $\mathbb{C}^{\mathbb{Z}^2}$ and hence need to consider functions $\phi$ of the state variables $(f_k)_{k \in \mathbb{Z}^2}$.  We call $\phi$ a \textit{cylindrical test function} if $\phi$ is a smooth and compactly supported function of finitely many of the state variables $(f_k)$.  If $\phi$ is cylindrical, we define its gradient $\mathbb{D}\phi$ by $$\mathbb{D}\phi := \sum_{k \in \mathbb{Z}_0^2} (\mathbb{D}_k \phi) e_k := \sum_{k \in \mathbb{Z}_0^2} \frac{\partial\phi}{\partial f_k}e_k.$$

\subsubsection{Probabilistic Tools}  Given a real-valued stochastic process $(X_t)_{t \geq 0}$, its quadratic variation $([X]_t)_{t \geq 0}$ is defined to be
\begin{equation}
[X]_t := \lim_{\epsilon \to 0} \int_0^t \frac{1}{\epsilon}(X_{\tau + \epsilon} - X_\tau)^2 \, d\tau,
\end{equation}
where the convergence is uniform on compacts in probability.  If $X$ is a continuous martingale, then this definition of the quadratic variation recovers the usual semimartingale quadratic variation of $X$.

We also recall here the Doob Martingale Inequality.

\begin{theorem}
(c.f. \cite{Kuo})  Suppose that $(X_t)_{t \in [0, T]}$ is a real-valued submartingale.  Then
$$\mathbb{P}\left(\sup_{t \in [0, T]} X_t \geq M \right) \leq \frac{\mathbb{E}(X_T \vee 0)}{M}.$$
Accordingly, for $1 < p < \infty$,
$$\|X_T\|_{L^p(\mathbb{P})} \leq \left\|\sup_{t \in [0, T]} X_t \right\|_{L^p(\mathbb{P})} \leq \frac{p}{p - 1} \|X_T\|_{L^p(\mathbb{P})}.$$
\end{theorem}

We will also have occasion to use the Burkholder-Davis-Gundy inequality:

\begin{theorem}
(c.f. \cite{strookbook})  If $M_t$ is a local martingale with $M_0 = 0$ and we denote $S_T = \max_{t \in [0, T]} |M_t|$, then for any $1 \leq p < \infty$ there is a constant $C_p$ for which
\begin{equation}
C_p^{-1}\mathbb{E}([M]_t^\frac{p}{2}) \leq \mathbb{E}(S_t^p) \leq C_p \mathbb{E}([M]_t^\frac{p}{2}).
\end{equation}
\end{theorem}

\section{Streamline formulation with small viscosity and noise}

We note that the mean of $\theta$ is constant in time, and so without loss of generality we only study solutions with mean zero.  This will help to avoid technical problems with low frequencies.
Regarding $\theta$ as analogous to the scalar vorticity in the 2D Euler equation, we find it advantageous to introduce the streamline $\psi = |D|^{-1}\theta$.  The streamline $\psi$ then satisfies the following equation:

\begin{equation}\label{SQGstreamline}
d\psi_t + |D|^{2\delta}\psi_t \,dx + |D|^{-1}(\nabla^\perp \psi_t \cdot \nabla)|D|\psi_t\,dx + |D|^{\delta - 1}dW_t = 0.
\end{equation}

Introduce the following abbreviated notation for the nonlinear term:

\begin{equation}
B(\psi, \psi^\prime) := -|D|^{-1}(\nabla^\perp \psi \cdot \nabla)|D|\psi^\prime.
\end{equation}

Recall the $L^2(\mathbb{T}^2)$ orthonormal basis $(e_k)_{k \in \mathbb{Z}^2}$ and expand $\psi = \sum_k \psi_k e_k$ and $B(\psi) = \sum_k B_k(\psi) e_k$.  Using the symmetry in the following sums in $h_1$ and $h_2$ as well as the identity $(h_2^\perp \cdot h_1) = (h_2^{\perp\perp} \cdot h_1^\perp) = -(h_1^\perp \cdot h_2)$ gives

\begin{align}\label{NonlinearityComponent}
B_k(\psi) & = \sum_{h_1 + h_2 = k} -|k|^{-1}(h_1^\perp \cdot h_2)|h_2| \psi_{h_1} \psi_{h_2} \notag \\
& = \frac12 \sum_{h_1 + h_2 = k} -|k|^{-1}(h_1^\perp \cdot h_2)|h_2|\psi_{h_1} \psi_{h_2} - |k|^{-1}(h_2^\perp \cdot h_1)|h_1| \psi_{h_1} \psi_{h_2} \\
& = \frac12 \sum_{h_1 + h_2 = k} |k|^{-1}(h_1^\perp \cdot h_2)(|h_1| - |h_2|) \psi_{h_1} \psi_{h_2}. \notag
\end{align}

For brevity we introduce the coefficients

\begin{equation}
\alpha_{h_1, h_2, k} := |k|^{-1}(h_1^\perp \cdot h_2)(|h_1| - |h_2|)
\end{equation}

so that

\begin{equation}
B_k(\psi) = \sum_{h_1 + h_2 = k} \alpha_{h_1, h_2, k} \psi_{h_1} \psi_{h_2}.
\end{equation}

By construction we have $\alpha_{h_1, h_2, k} = \alpha_{h_2, h_1, k}$, and $B_{-k}(\psi, \psi) = B_k(\psi, \psi)$.  Therefore when regarded as a vector field in frequency, $B(\psi, \psi)$ is formally divergence-free:
$$\text{div}(B(\psi, \psi)) = \sum_k kB_k(\psi, \psi) = \sum_k -kB_{-k}(\psi, \psi) = -\text{div}(B(\psi, \psi)).$$

We introduce the following Gaussian measure $\rho$ which is built using the conserved $H^1$ norm of the streamline function\footnote{This is equivalent to the $L^2$ norm of $\theta$.}

\begin{equation}
\int e^{i\langle f, \psi\rangle} d\rho(f) = e^{-\frac12\|\psi\|_{H^1}^2} = \exp\left(-\frac12 \sum_k |k|^2 |\psi_k|^2\right),
\end{equation}

and is also characterized by the integration by parts formula\footnote{Here the appearance of the factor $\psi_{-k}$ is due to the fact that in the integration by parts, the only terms in the exponential that appear in the calculation are $$\frac12|\psi_k|^2 + \frac12|\psi_{-k}|^2 = \psi_k\psi_{-k},$$ where we have also used the reality condition $\overline{\psi}_k = \psi_{-k}$.}

\begin{equation}
\int \mathbb{D}_k F(\psi) \, d\rho(\psi) = \int |k|^2\psi_{-k} F(\psi) \, d\rho(\psi).
\end{equation}

First introduce the Ornstein-Uhlenbeck operator associated to our equation, defined for any cylindrical function $\phi$ of the stream function $\psi$:

\begin{equation}
L_0\phi(\psi) := \sum_k |k|^{2\delta} (|k|^2 \psi_k \mathbb{D}_k \phi(\psi) + \mathbb{D}_{k}\mathbb{D}_{-k} \phi(\psi)).
\end{equation}

We also introduce the Mallevin-Sobolev norm

\begin{equation}
\mathcal{E}^\delta(\phi)(\psi) := \frac12 \sum_k |k|^{2\delta} |\mathbb{D}_k \phi(\psi)|^2.
\end{equation}

The expression $\mathcal{E}^\delta(\phi)(\psi)$ is chosen with respect to the Gaussian measure $\rho$ so that the following integration by parts formula holds: 
$$\int \phi(\psi) L_0\phi(\psi) \, d\mu(\psi) = \int \mathcal{E}^\delta(\phi)(\psi) \, d\mu(\psi).$$

For further details consult \cite{gpbrazil}.

\section{The Truncated Systems  : Construction and Global Existence}

We will construct solutions to \eqref{SQGstreamline} by extracting a candidate from a sequence of solutions to a frequency truncated version of \eqref{SQGstreamline}.  To that end, for $N \geq 1$ denote
\begin{equation}
\Pi_N(f) = \sum_{|k| \leq N} f_k e_k
\end{equation}
and the truncated nonlinearity by $$B^N(f, g) = \Pi_N B(\Pi_N f, \Pi_N g).$$  We introduce
\begin{equation}\label{PsiNEquation}
\begin{cases}d\psi^{(N)} = -|D|^{2\delta}\psi^{(N)}dt + B_N(\psi^{(N)})dt - |D|^{\delta - 1}dW_t \\ \psi^{(N)}(0) = \psi_0,
\end{cases}
\end{equation}
where the initial data $\psi_0$ has distribution $\rho$.  Then we can decompose the approximate solution $$\psi^{(N)} = \Pi_N\psi^{(N)} + (I - \Pi_N)\psi^{(N)} := v^{(N)} + z^{(N)}$$ so that each component of $z^{(N)}$ is an Ornstein-Uhlenbeck process satisfying
$$
\begin{cases}
dz^{(N)} = -|D|^{2\delta}z^{(N)} dt - (I - \Pi_N)|D|^{\delta - 1}dW_{t}\\
z^{(N)}(0) = (I - \Pi_N)\psi_0,
\end{cases}
$$
where $\psi_0 = \sum_k \psi_{0, k}e_k$, and $v^{(N)}$ satisfies the finite dimensional system
\begin{equation}\label{FiniteDimensionalSystem}
\begin{cases}
dv^{(N)} = -|D|^{2\delta}v^{(N)}dt + B^N(v^{(N)})dt - |D|^{\delta - 1}\Pi_N dW_t \\
v^{(N)}_0 = \Pi_N \psi_0.
\end{cases}
\end{equation}
One can then explicitly check that $z^{(N)}(e_k) = \psi_0(e_k)$ has marginal distribution $\rho_k$; in particular the $z_k$ are independent and stationary.  At this point we can show that solutions to the truncated system are global.

\begin{proposition}
For each $N \geq 1$ and $T > 0$, the finite dimensional processes $v^{(N)}$ solving \eqref{FiniteDimensionalSystem} is almost surely defined on $[0, T]$.
\end{proposition}

\begin{proof}
We proceed by energy estimate and calculate using Ito's Lemma that
\begin{align*}
\frac12 d(\|v^{(N)}(t)\|_{L^2}^2) & = -\||D|^\delta v^{(N)}_k(t)\|_{L^2}^2 \, dt + \langle v^{(N)}_k(t), B_N(v^{(N)}_k(t))\rangle_{L^2} \, dt \\
& - \langle v^{(N)}_k(t), |D|^{\delta - 1}\Pi_N dW_t \rangle_{L^2} + \langle\langle |D|^{\delta - 1}W \rangle\rangle_t \, dt.
\end{align*}
As $\langle v^{(N)}_k(t), B_N(v^{(N)}_k(t)) \rangle_{L^2} = 0$, we have
\begin{align*}
\frac12 d(\|v^{(N)}(t)\|_{L^2}^2) \leq \langle v^{(N)}(t), |D|^{\delta - 1}\Pi_N dW_t \rangle_{L^2} + \sum_{|k| \leq N} |k|^{2\delta - 2} \, dt.
\end{align*}
Integrate in time on some interval $[0, t]$ on which $v^{(N)}$ exists; then we have
\begin{align*}
\frac12 \|v^{(N)}(t)\|_{L^2}^2 \leq \frac12 \|v^{(N)}(0)\|_{L^2}^2 + \int_0^t \langle v^{(N)}(\tau), |D|^{\delta - 1}\Pi_N dW_\tau \rangle_{L^2} + t\sum_{|k| \leq N} |k|^{2\delta - 2}.
\end{align*}
Then for $t \in [0, T]$ we have using Doob's martingale maximal inequality that
\begin{align*}
\mathbb{E}_\rho\left[\sup_{t \in [0, T]} \|v^{(N)}(t)\|_{L^2}^4\right] & \leq 2 \mathbb{E}_\rho[\|v^{(N)}(0)\|_{L^2}^4] + 4\mathbb{E}_\rho\left[\sup_{t \in [0, T]} \left(\int_0^t \sum_k v^{(N)}_k|k|^{\delta - 1} dW_t\right)^2\right] + 4N^{4\delta}T^2 \\
& \leq 2 \mathbb{E}_\rho[\|v^{(N)}(0)\|_{L^2}^4] + 16\mathbb{E}_\rho\left[\left(\int_0^T \sum_k v^{(N)}_k|k|^{\delta - 1} dW_t\right)^2\right] + 4N^{4\delta}T^2 \\
& = 2 \mathbb{E}_\rho[\|v^{(N)}(0)\|_{L^2}^4] + 16\int_0^T \mathbb{E}_\rho[\||D|^{\delta - 1} v^{(N)}(t)\|_{L^2}^2] \, dt + 4N^{4\delta}T^2 \\
& \leq 2 \mathbb{E}_\rho[\|v^{(N)}(0)\|_{L^2}^4] + 16\int_0^T \mathbb{E}_\rho[\|v^{(N)}(t)\|_{L^2}^2] \, dt + 4N^{4\delta}T^2,
\end{align*}
and this a priori bound is sufficient to extend the solution almost surely to $[0, T]$ for arbitrarily chosen $T$.
\end{proof}

\section{Infinitesimal Invariance of the Gaussian Measure under the Truncated Flow}

We would also like to show that each solution $\psi^{(N)}$ to \eqref{PsiNEquation} has stationary distribution $\rho$.  The full flow corresponding to the truncated mSQG above is
\begin{equation}
L_N \phi(\psi) := L_0 \phi(\psi) + \langle B^N(\psi), \nabla \phi(\psi) \rangle_{H^1},
\end{equation}
where $L_0$ is the untruncated Orenstein-Uhlenbeck generator
\begin{equation}
L_0\phi(\psi) := \sum_k |k|^{2\delta} (|k|^2 \psi_k \mathbb{D}_k \phi(\psi) + \mathbb{D}_{k}\mathbb{D}_{-k} \phi(\psi)).
\end{equation}

\begin{proposition}
The Gaussian measure $\rho$ is infinitesimally invariant under $L_N$.  That is, for every cylindrical test function $\psi$, $\mathbb{E}_\rho[L_N \phi (\psi)] = 0$.
\end{proposition}

\begin{proof}
It suffices to show that $\mathbb{E}_\rho(L_0\phi(\psi)) = 0$ and $\mathbb{E}_\rho(\langle B_n(\psi), \nabla \phi(\psi) \rangle) = 0$.  

First we show that $\mathbb{E}_\rho(L_0\phi(\psi)) = 0$.  Without loss of generality we may take $\phi = \phi(\psi_1, \ldots, \psi_m) = \phi(\psi^{(m)})$; then we have with integration by parts that
\begin{align*}
\mathbb{E}_\rho(L_0\phi(\psi)) & = \int L_0\phi(\psi) \, d\rho_m(\psi^{(m)}) \\
& = \int \sum_{|k| \leq m} |k|^{2\delta} \left( |k|^2 \psi_k \mathbb{D}_k \phi(\psi) + \mathbb{D}_{k}\mathbb{D}_{-k} \phi(\psi) \right) \, e^{-\frac12\sum_{|j| \leq m} |j|^2\,|\psi_j|^2} d\psi_1 \cdots d\psi_m \\
& = \int \sum_{|k| \leq m} |k|^{2\delta} \left( |k|^2 \psi_k \mathbb{D}_k \phi(\psi) + (-|k|^2\psi_k) \mathbb{D}_{k}\phi(\psi) \right) \, e^{-\frac12\sum_{|j| \leq m} |j|^2\,|\psi_j|^2} d\psi_1 \cdots d\psi_m \\
& = 0.
\end{align*}

Next we show that $\mathbb{E}_\rho(\langle B_N(\psi), \nabla \phi(\psi) \rangle) = 0$.  We have
\begin{equation}
\int \text{div}_\rho(B_N(\psi))\phi(\psi) \, d\rho_m(\psi^{(m)}) = \int \langle B_N(\psi), \nabla \phi(\psi) \rangle_{H^1} \, d\rho_m(\psi^{(m)})
\end{equation}
by the definition of the cylindrical Gaussian function $\rho$.  We calculate explicitly that
\begin{align*}
\int \langle B_N(\psi), \nabla \phi(\psi) \rangle_{H^1} \, d\rho_m(\psi^{(m)})
& = \int \langle |D|^2 B_N(\psi), \nabla \phi(\psi) \rangle_{L^2} \, d\rho_m(\psi^{(m)}) \\
& = \int \langle B_N(\psi), |D|^2 \nabla \phi(\psi) \rangle_{L^2} \, e^{-\frac12\sum_{|j| \leq m} |j|^2\,|\psi_j|^2} d\psi_1\cdots d\psi_m \\
& = \int \langle B_N(\psi) e^{-\frac12\sum_{|j| \leq m} |j|^2\,|\psi_j|^2}, \nabla(|D|^2 \phi(\psi)) \rangle_{L^2} \, d\psi_1\cdots d\psi_m \\
& = \int \text{div}\left(B_N(\psi) e^{-\frac12\sum_{|j| \leq m} |j|^2\,|\psi_j|^2}\right) |D|^2 \phi(\psi) \, d\psi_1\cdots d\psi_m,
\end{align*}
where $\text{div}$ now denotes the divergence with respect to the standard inner product on $\mathbb{R}^{(2N + 1)^2}$.  But now
\begin{align*}
& \;\quad \text{div}\left(B_N(\psi) e^{-\frac12\sum_{|j| \leq m} |j|^2\,|\psi_j|^2}\right) \\
& = \text{div}(B_N(\psi)) e^{-\frac12\sum_{|j| \leq m} |j|^2\,|\psi_j|^2} + B_N(\psi) \cdot \nabla\left(e^{-\frac12\sum_{|j| \leq m} |j|^2\,|\psi_j|^2}\right) \\
& = \text{div}(B_N(\psi)) e^{-\frac12\sum_{|j| \leq m} |j|^2\,|\psi_j|^2} - \langle B_N(\psi), |k|^2 \psi \rangle \left(e^{-\frac12\sum_{|j| \leq m} |j|^2\,|\psi_j|^2}\right).
\end{align*}
This last expression vanishes entirely, as the first term vanishes since $\text{div}(B_N(\psi)) = 0$, and the second term vanishes since $\langle |D|B_N(\psi), |D|\psi \rangle = 0$.
\end{proof}

\section{Definition of Energy Solutions and Outline of Argument}

We will use the above sequence of approximate solutions $(\psi^{(N)})$ to construct candidate solutions to the streamline SQG equation \eqref{SQGstreamline}.  The sense of solution that we will pursue here is the \textit{energy solution} used in \cite{gj}.  Such a solution is a particular kind of process called a \textit{controlled process}, so named since such processes can be regarded as a perturbations of the Orenstein-Uhlenbeck process generated by the linearization of \eqref{SQGstreamline}.

\begin{definition}\label{ControlledProcessDefinition}
(Controlled processes to SQG).  (c.f. \cite{gj, gpbrazil})  We call $\psi$ a \textit{controlled process} on $[0, T]$ to \eqref{SQGstreamline} if it satisfies the following properties:
\begin{itemize}
\item[(1)]{The distribution of $\psi_t$ is the Gaussian measure $\rho$.}
\item[(2)]{There exists a stochastic process $\mathcal{A} \in C([0, T] : \mathscr{S}^\prime)$ of zero quadratic variation for which $\mathcal{A}_0 = 0$ at time $t = 0$ and for which the process $M^+$ defined for each cylindrical test function $\phi$ by
\begin{equation}\label{ForwardMartingale}
M_t^+(\phi) := \psi_t(\phi) - \psi_0(\phi) - \int_0^t \psi_\tau(-|D|^{2\delta}\phi) \, d\tau - \mathcal{A}_t(\phi)
\end{equation}
is a martingale with respect to the filtration generated by $\psi$ and with quadratic variation $[M^+(\phi)]_t = t\|\phi\|_{H^{1 - \delta}}^2$.}
\item[(3)]{If $\psi_t$ and $\mathcal{A}_t$ in \eqref{ForwardMartingale} are replaced by the time reversed processes $
\psi_{T - t}$ and $-\mathcal{A}_{T - t}$ respectively, then the corresponding process $M^-$ is also a martingale with respect to the filtration generated by $u_{-t}$ with quadratic variation $[M^-(\phi)]_t = t\|\phi\|_{H^{1 - \delta}}^2$.}
\end{itemize}
We denote the set of all controlled processes by $\mathscr{R}_\delta(0, T)$.
\end{definition}

Using the sequence of approximate solutions from the last section, we will show in the sequel that for any $\delta > 0$, the sequence
$$\int_0^t B^N(\psi^{(N)}_\tau, \psi^{(N)}_\tau)\, d\tau$$
converges in a suitable Fourier-Lebesgue space almost surely as $N \to \infty$.\footnote{Although it might appear that this limit depends on the particular method of frequency truncation we have taken here, it is not difficult to show that the limit resulting from a sequence of truncations consisting of any compact exhaustion of $\mathbb{Z}_0^2$ agree.  Hence any reasonable Galerkin approximation to \eqref{SQGstreamline} recovers the same notion of energy solution.}  We denote this limit by $$\int_0^t B(\psi_\tau, \psi_\tau)\, d\tau.$$ Then we can make precise our notion of energy solution:

\begin{definition}\label{EnergySolutionDefinition}
(Energy Solution).  (c.f. \cite{gj, gpbrazil})  A controlled process $\psi \in \mathscr{R}_\delta(0, T)$ is an \textit{energy solution} to \eqref{SQGstreamline} if almost surely the following compatibility condition holds for all cylindrical test functions $\phi$:
\begin{equation}
\mathcal{A}_t(\phi) = \left\langle \int_0^t B(\psi_\tau, \psi_\tau)\,d\tau, \phi \right\rangle.
\end{equation}
\end{definition}

\section{The Key Cancellation of Controlled Processes}

The controlled processes ansatz allows us to reduce estimates of the time-integral of the nonlinearity of \eqref{SQGstreamline} to estimates on martingales, for which we have the tools to provide better estimates.

Let $h : [0, T] \times \Pi_N(L^2) \to \mathbb{C}$ be a smooth cylindrical function which is intended to stand in for functions of the frequency truncated nonlinearities appearing in \eqref{PsiNEquation}.  

Suppose that $\psi \in \mathscr{R}_\delta(0, T)$ is given.  We would like to derive an expression that expands the time-integrated nonlinearity using the definition of controlled processes.  To do so, we use Ito's Lemma and property (2) of the Definition \ref{ControlledProcessDefinition} of controlled processes allows us to express $h(\Pi_N \psi_t, t)$ as 
\begin{equation}
h(\Pi_N \psi_t, t)= h(\Pi_N \psi_0, 0) + \int_0^t (\partial_\tau + L_0^{(N)})h(\Pi_N \psi_\tau, \tau) \, d\tau + \int_0^t \mathbb{D}h(\Pi_N \psi_\tau, \tau) \cdot d(\Pi_N\mathcal{A}_\tau) + \mathcal{M}^+_t,
\end{equation}
where here
\begin{equation}
L_0^{(N)}\phi(\psi) := \sum_{|k| \leq N} |k|^{2\delta} (|k|^2 \psi_k \mathbb{D}_k \phi(\psi) + \mathbb{D}_{k}\mathbb{D}_{-k} \phi(\psi)),
\end{equation}
where the martingale $\mathcal{M}^+$ has quadratic variation
\begin{equation}
[\mathcal{M}^+]_t = \int_0^t (\mathcal{E}^\delta)^{(N)}(h(\cdot, \tau))(\Pi_N\psi_\tau)\,d\tau,
\end{equation}
and where
\begin{equation}
(\mathcal{E}^\delta)^{(N)}(\phi)(\psi) := \frac12 \sum_{|k| \leq N} |k|^{2\delta} |\mathbb{D}_k \phi(\psi)|^2.
\end{equation}
Similarly, using property (3) of the definition of controlled processes and proceeding in the same way, we arrive at the relation
\begin{align*}
h(\Pi_N \psi_{T - t}, T - t) & = h(\Pi_N \psi_T, T) + \int_0^t (-\partial_\tau + L_0^{(N)})h(\Pi_N \psi_{T - \tau}, T - \tau) \, d\tau \\
& \qquad - \int_0^t \mathbb{D}h(\Pi_N \psi_{T - \tau}, T - \tau) \cdot d(\Pi_N\mathcal{A}_{T - \tau}) + \mathcal{M}^-_t,
\end{align*}
where
\begin{equation}
[\mathcal{M}^-]_t = \int_0^t (\mathcal{E}^\delta)^{(N)}(h(\cdot, T - \tau))(\Pi_N\psi_{T - \tau})\,d\tau.
\end{equation}
We can express the time derivative of the nonlinearity in terms of the above martingales alone by writing
\begin{equation}\label{KeyRepresentation}
2\int_0^t L_0^{(N)} h(\Pi_N\psi_\tau, \tau) \, d\tau = -\mathcal{M}^+_t + \mathcal{M}^-_{T - t} - \mathcal{M}^-_T.
\end{equation}
This is the key representation which permits us control of the nonlinearity using martingale estimates.

\section{Martingale Estimates and the Ito Trick for Controlled Processes}

We now show how martingale estimates permit good control over the time integral of the nonlinearity.

\begin{proposition}
(The ``Ito Trick'').  Let $h : [0, T] \times \Pi_N(L^2)$ be a cylindrical function and let $\psi \in \mathscr{R}_\delta(0, T)$ be given.  Then for any $p \geq 1$,
\begin{equation}
\left\|\sup_{t\in[0, T]} \left|\int_0^t L_0h(\Pi_n\psi_\tau, \tau)\, d\tau \right|\right\|_{L^p(\rho)} \leq C_pT^\frac12\sup_{t \in [0, T]}\|\mathcal{E}^\delta(h(\cdot, t))\|^\frac12_{L^\frac{p}{2}(\rho)}.
\end{equation}
If in addition $h$ takes the form $h(\psi, t) = e^{a(T - t)}\tilde{h}(\psi)$ for some $a \in \mathbb{R}$, then the above estimate can be improved to read
\begin{equation}
\left\|\sup_{t\in[0, T]} \left|\int_0^t e^{a(T - \tau)}L_0\tilde{h}(\Pi_n\psi_\tau)\, d\tau \right|\right\|_{L^p(\rho)} \leq C_p\left(\frac{1 - e^{2aT}}{2a}\right)^\frac12 \sup_{t \in [0, T]}\|\mathcal{E}^\delta(h(\cdot, t))\|^\frac12_{L^\frac{p}{2}(\rho)}.
\end{equation} 
\end{proposition}

\begin{proof}
By using the key representation \eqref{KeyRepresentation}, the Burkholder-Davis-Gundy inequality, and our expressions for the quadratic variations of the martingales $\mathcal{M}^+$ and $\mathcal{M}^-$ we find that
\begin{align*}
\left\|\sup_{t\in[0, T]} \left|\int_0^t 2L_0h(\Pi_n\psi_\tau, \tau)\, d\tau \right|\right\|_{L^p(\rho)}& \leq \left\|\sup_{t\in[0, T]}|\mathcal{M}^+_t|\right\|_{L^p(\rho)} + 2 \left\|\sup_{t\in[0, T]}|\mathcal{M}^-_t|\right\|_{L^p(\rho)} \\
& \leq C_p\left(\|[M^+]_t\|^\frac12_{L^\frac{p}{2}(\rho)} + \|[M^+]_t\|^\frac12_{L^\frac{p}{2}(\rho)}\right) \\
& \leq C_p\left\|\int_0^t \mathcal{E}^\delta(h(\cdot, \tau))(\psi_\tau) \,d\tau\right\|^\frac12_{L^\frac{p}{2}(\rho)} \\
& \leq C_p\left(\int_0^t \|\mathcal{E}^\delta(h(\cdot, \tau))(\psi_\tau)\|_{L^\frac{p}{2}(\rho)} \,d\tau\right)^\frac12.
\end{align*}
For the more general estimate we simply estimate so that $$\left(\int_0^t \|F(\tau)\| d\tau\right)^\frac12 \leq T^\frac12 \sup_{[0, T]} \|F\|,$$ and in the more specific case $F(\psi, t) = e^{a(T - t)}\tilde{F}(\psi)$ we estimate as in $$\left(\int_0^t e^{a(T - \tau)}\|\tilde{F}(\tau)\| d\tau\right)^\frac12 \leq \left(\int_0^t e^{a(T - \tau)}\,d\tau\right)^\frac12 \sup_{[0, T]} \|\tilde{F}\|.$$
\end{proof}

\section{Application to Estimating the SQG Nonlinearity}

In the previous section we have shown that, by restricting to studying controlled solutions, we may control quantities of the form $L_0 h(\psi)$ in terms of the expectation of $\mathcal{E}^\delta(h(\psi))$.  This motivates us to search for some quantity $H^N$ for which we can write $L_0 H^N = B^N$.  Since such an $H^N$ satisfies a Poisson-like equation with forcing term $B^N$, we expect $H^N$ to be more regular than $B^N$.  Motivated by \cite{gj}, we choose

\begin{equation}
H^N(\psi) = -\int_0^\infty B^N(e^{-|D|^{2 + 2\delta}t} \psi) \, dt.
\end{equation}

We now verify that with this definition, we have indeed that

\begin{proposition}
$L_0 H^N(\psi) = B^N(\psi)$.
\end{proposition}

\begin{proof}
We first check the identity $\mathbb{D}_k \mathbb{D}_{-k} B^N(e^{-|D|^{2 + 2\delta}t} \psi) = 0$ as follows:  for $|k| \leq N$ we have
\begin{align*}
& \quad\; \mathbb{D}_k \mathbb{D}_{-k} B^N(e^{-|D|^{2 + 2\delta}t} \psi) \\
& = -\mathbb{D}_k \mathbb{D}_{-k} \Pi_N|D|^{-1}(\nabla^\perp \exp(-|D|^{2 + 2\delta}t) \psi^{(N)} \cdot \nabla)|D| \exp({-|D|^{2 + 2\delta}t}) \psi^{(N)} \\
& = \mathbb{D}_k \Biggl( |D|^{-1}((ik)^\perp \exp(-|k|^{2 + 2\delta}t) e_{-k} \cdot \nabla)|D| \exp({-|D|^{2 + 2\delta}t}) \psi^{(N)} \\
& \qquad \; + |D|^{-1}(\nabla^\perp \exp(-|D|^{2 + 2\delta}t) \psi^{(N)} \cdot ik)|k| \exp({-|k|^{2 + 2\delta}t}) e_{-k} \Biggr) \\
& = |D|^{-1}((ik)^\perp \exp(-|k|^{2 + 2\delta}t) e_{-k} \cdot (ik))|k| \exp({-|k|^{2 + 2\delta}t}) e_k \\
& \; + |D|^{-1}((ik)^\perp \exp(-|k|^{2 + 2\delta}t) e_k \cdot (ik))|k| \exp({-|k|^{2 + 2\delta}t}) e_{-k} \\
& = 2|D|^{-1}|k|\exp(-2|k|^{2 + 2\delta}t)(k^\perp \cdot k) \\
& = 0,
\end{align*}
and for $|k| > N$ the expression is 0 immediately.
This simplifies our verification of the fact that $L_0 H^N = B^N$.  Indeed, writing $S_\delta(t) = e^{-|D|^{2 + 2\delta}t}$, we have
\begin{align*}
& \quad \; L_0 H^N(\psi) \\
& = -\int_0^\infty L_0 B^N(S_\delta(t)\psi) \, dt \\
& = \int_0^\infty \sum_{|k| \leq N} |k|^{2 + 2\delta} S_\delta(t)\psi_k \mathbb{D}_k B^N(S_\delta(t)\psi)  \, dt \\
& = \int_0^\infty \sum_{|k| \leq N} |k|^{2 + 2\delta} S_\delta(t)\psi_k \sum_{|j| \leq N} \mathbb{D}_k B_j(S_\delta(t)\psi) e_j \, dt\\
& = \int_0^\infty \sum_{|k| \leq N} |k|^{2 + 2\delta} S_\delta(t)\psi_k \sum_{|j| \leq N} \sum_{|j_1|, |j_2| \leq N}^{j_1 + j_2 = j} \alpha_{j_1, j_2, j} \mathbb{D}_k(S_\delta(t)\psi_{j_1} S_\delta(t)\psi_{j_2}) e_j \, dt\\
& = \int_0^\infty \sum_{|k| \leq N} |k|^{2 + 2\delta} S_\delta(t)\psi_k \sum_{|j| \leq N} 2 \alpha_{k, 
j - k, j} S_\delta(t) \psi_{j - k}^{(N)} e_j \, dt \\
& = \int_0^\infty \sum_{|j| \leq N} \left( \sum_{|k| \leq N} 2|k|^{2 + 2\delta} \alpha_{k, j - k, j} S_\delta(t)\psi_k^{(N)} S_\delta(t)\psi_{j - k}^{(N)}\right) e_j \, dt\\
\end{align*}
\begin{align*}
& = -\int_0^\infty \sum_{|j| \leq N} \Biggl( \sum_{|k| \leq N} \alpha_{k, j - k, j} (-|k|^{2 + 2\delta}) S_\delta(t)\psi_k^{(N)} S_\delta(t)\psi_{j - k}^{(N)} \\
& \hspace{3cm} + \sum_{|k| \leq N} \alpha_{k, j - k, j} S_\delta(t)\psi_k^{(N)} (-|j - k|^{2 + 2\delta}) S_\delta(t)\psi_{j - k}^{(N)} \Biggr) e_j \, dt\\
& = -\int_0^\infty \frac{d}{dt} B^N(S_\delta(t)\psi) \, dt \\
& = B^N(S_\delta(0)\psi) - \lim_{t \to \infty}B^N(S_\delta(t)\psi) \\
& = B^N(\psi).
\end{align*}
We note that in in the above derivation we have used the identity
\begin{equation*}
\mathbb{D}_k e^{-|D|^{2 + 2\delta}t}\psi_k = e^{-|D|^{2 + 2\delta}t}\mathbb{D}_k \psi_k = e^{-|D|^{2 + 2\delta}t}1 = 1,
\end{equation*}
which holds in the sense of tempered distributions.
\end{proof}

We use this relation to provide exponential type bounds on the Malliavin-Sobolev norm appearing in the Ito trick.

\begin{proposition}\label{ExponentialBounds}
Fix $N > 1$ and $k$ with $|k| \leq N$.  There is a $\lambda > 0$ sufficiently small depending on $\delta > 0$ so that
$$\mathbb{E}_\rho \exp\left(\lambda |k|^{2\delta} \mathcal{E}^\delta(H_k^{N, \pm})(\psi)\right)$$ is uniformly bounded in $k$ and $N$.  Moreover if $M < N$, then there is a $\lambda > 0$ sufficiently small depending on $\delta > 0$ so that $$\mathbb{E}_\rho \exp\left(\lambda |k|^{2\delta} \mathcal{E}^\delta(H_k^{N, \pm} - H_k^{M, \pm})(\psi)\right)$$ is uniformly bounded in $k$, $N$, and $M$.
\end{proposition}

\begin{proof}
Directly manipulating the definition of $H^N$ gives the following explicit expression:

\begin{align*}
H^N(\psi) & := -\int_0^\infty B^N(e^{-|D|^{2 + 2\delta}t} \psi) \, dt \\
& = -\int_0^\infty \sum_{|k| \leq N} B_k(e^{-|D|^{2 + 2\delta}t} \psi^{(N)})e_k \, dt \\
& = -\int_0^\infty \sum_{|k| \leq N} \sum_j \alpha_{k, j, k - j} e^{-(|j|^{2 + 2\delta} + |k - j|^{2 + 2\delta})t} \psi_j^{(N)} \psi_{k - j}^{(N)} e_k \, dt \\
& = -\sum_{|k| \leq N} \sum_j \alpha_{k, j, k - j} \int_0^\infty e^{-(|j|^{2 + 2\delta} + |k - j|^{2 + 2\delta})t} \, dt \, \psi_j^{(N)} \psi_{k - j}^{(N)} e_k \\
& = -\sum_{|k| \leq N} \sum_j \frac{\alpha_{k, j, k - j}}{|j|^{2 + 2\delta} + |k - j|^{2 + 2\delta}} \psi_j^{(N)} \psi_{k - j}^{(N)} e_k \\
& = \sum_{|k| \leq N} \sum_j \frac{|k|^{-1}(j^\perp \cdot (k - j))(|k - j| - |j|)}{|j|^{2 + 2\delta} + |k - j|^{2 + 2\delta}} \psi_j^{(N)} \psi_{k - j}^{(N)} e_k \\
& = \sum_{|k| \leq N} \sum_{j_1 + j_2 = k} \frac{|k|^{-1}(j_1^\perp \cdot j_2)(|j_2| - |j_1|)}{|j_1|^{2 + 2\delta} + |j_2|^{2 + 2\delta}} \psi_{j_1}^{(N)} \psi_{j_2}^{(N)} e_k. \\
\end{align*}

Now denote by $H^N(\psi)_k^\pm$ the real and imaginary parts of the $k$-th component of $H(\psi)$, respectively.  If we adopt the convention $i^+ = 1$ and $i^- = i$, then since both $\psi$ and $H^N(\psi)$ are real-valued we have for $|k| \leq N$ that
\begin{align*}
H^N(\psi)_k^\pm & = \frac{H^N(\psi)_k \pm \overline{H^N(\psi)_k}}{2i^\pm} \\
& = \frac{H^N(\psi)_k \pm H^N(\psi)_{-k}}{2i^\pm} \\
& = \frac{1}{2|k|i^\pm} \sum_{j_1 + j_2 = k} \frac{(j_1^\perp \cdot j_2)(|j_2| - |j_1|)}{|j_1|^{2 + 2\delta} + |j_2|^{2 + 2\delta}} (\psi_{j_1}^{(N)} \psi_{j_2} ^{(N)} \pm \psi_{-j_1} ^{(N)} \psi_{-j_2} ^{(N)}).
\end{align*}
We recall that our martingales are bounded by quantities of the form
\begin{equation}
\mathcal{E}^\delta(H_k^\pm)(\psi) := \frac12 \sum_{j \in \mathbb{Z}} |j|^{2\delta} |\mathbb{D}_j H(\psi)_k^\pm|^2;
\end{equation}
we now expand this quantity explicitly using the above characterization of $H^N$.  We have first that
\begin{align}\label{DerivOfH}
\mathbb{D}_j(H^N(\psi)^\pm_k) & = \frac{1}{|k|i^\pm} \left(\frac{(j^\perp \cdot (k - j))(|k - j| - |j|)}{|j|^{2 + 2\delta} + |k - j|^{2 + 2\delta}} \psi_{k - j} ^{(N)} \pm \frac{(j^\perp \cdot (-k - j))(|k + j| - |j|)}{|j|^{2 + 2\delta} + |k + j|^{2 + 2\delta}} \psi_{-k - j} ^{(N)}\right) \\
& =: \frac{1}{|k|i^\pm} \left(\beta_{k, j} \psi_{k - j} ^{(N)} \pm 
\beta_{k, -j} \psi_{-k - j} ^{(N)}\right). \notag
\end{align}
Using the symmetries $\beta_{k, -j} = \beta_{k, j}$ and $|\psi_{l}|^2 = |\psi_{-l}|^2$, we have
\begin{align}\label{EOfDerivOfH}
\mathcal{E}^\delta(H_k^\pm)(\psi) & = \frac12 \sum_{j \in \mathbb{Z}^2} |j|^{2\delta} |\mathbb{D}_j(H^N(\psi)^\pm_k)|^2 \notag \\
& \leq \sum_{j \in \mathbb{Z}^2} \pm \frac{|j|^{2\delta}}{|k|^2} \left( \beta_{k, j}^2 |\psi_{k - j}^{(N)}|^2 + \beta_{k, -j}^2 |\psi_{-k - j} ^{(N)} |^2\right) \notag \\
& = 2\sum_{j \in \mathbb{Z}^2} \pm \frac{|j|^{2\delta}}{|k|^2} \beta_{k, j}^2 |\psi_{k - j} ^{(N)} |^2 \notag \\
& = 2\sum_{j \in \mathbb{Z}^2} \pm \frac{|j|^{2\delta}}{|k|^2} \left(\frac{(j^\perp \cdot (k - j))(|k - j| - |j|)}{|j|^{2 + 2\delta} + |k - j|^{2 + 2\delta}}\right)^2 |\psi_{k - j}^{(N)}|^2\\
& = 2\sum_{j \in \mathbb{Z}^2} \pm \frac{|j|^{2\delta}}{|k|^2|k - j|^2} \left(\frac{(j^\perp \cdot (k - j))(|k - j| - |j|)}{|j|^{2 + 2\delta} + |k - j|^{2 + 2\delta}}\right)^2 |\psi_{k - j}^{(N)}|^2 |k - j|^2 \notag \\
& =: \sum_{j \in \mathbb{Z}^2} \gamma_{k, j} |\psi_{k - j}^{(N)}|^2 |k - j|^2. \notag
\end{align}
Notice that we may bound the coefficient $\gamma_{k, j}$ as follows:
\begin{align*}
\gamma_{k, j} & \precsim \frac{|j|^{2\delta}}{|k|^2|k - j|^2}\left(\frac{|j||k - j||k|}{|j|^{2 + 2\delta} + |k - j|^{2 + 2\delta}}\right)^2 \\
& \leq \frac{|j|^{2 + 2\delta}}{(|j|^{2 + 2\delta} + |k - j|^{2 + 2\delta})^2}. \\
\end{align*}
As a consequence we have 
\begin{align*}
\sum_{j \in \mathbb{Z}} \gamma_{k, j} & \leq \int_{\mathbb{R}^2} \frac{|x|^{2 + 2\delta}}{(|x|^{2 + 2\delta} + |k - x|^{2 + 2\delta})^2} \, dx \\
& = |k|^{-2\delta} \int_{\mathbb{R}^2} \frac{|y|^{2 + 2\delta}}{(|y|^{2 + 2\delta} + |k/|k| - y|^{2 + 2\delta})^2} \, dy \\
& \precsim |k|^{-2\delta},
\end{align*}
uniformly in $k$.  Using this, we now proceed to give exponential bounds for the quantity $\mathcal{E}^\delta(H_k^{N, \pm})(\psi)$.  We have
\begin{align}
& \quad\; \mathbb{E}_\rho \exp\left(\lambda |k|^{2\delta} \mathcal{E}^\delta(H_k^{N, \pm})(\psi)\right) \label{ExpectationToBound} \\
& \leq \mathbb{E}_\rho \exp\left(\lambda |k|^{2\delta} \sum_{j \in \mathbb{Z}^2} \gamma_{k, j} |\psi_{k - j}^{(N)}|^2|k - j|^2 \right) \\
& \leq \mathbb{E}_\rho \exp\left(C\lambda \frac{\sum_{j \in \mathbb{Z}^2} \gamma_{k, j} |\psi_{k - j}^{(N)}|^2|k - j|^2}{\sum_{j \in \mathbb{Z}^2} \gamma_{k, j}} \right) \label{UsedGammaSumEstimate} \\
& \leq \frac{\sum_{j \in \mathbb{Z}^2} \gamma_{k, j} \mathbb{E}_\rho(\exp(C\lambda|\psi_{k - j}^{(N)}|^2|k - j|^2))}{\sum_{j \in \mathbb{Z}^2} \gamma_{k, j}}, \label{UsedJensenHere}
\end{align}
where in \eqref{UsedGammaSumEstimate} we used the estimate on $\sum_j \gamma_{j, k}$, and in \eqref{UsedJensenHere} we used Jensen's inequality.  Now since $\psi_{k - j}$ is distributed according to our Gaussian measure, it has mean zero and variance $|k - j|^{-2}$, we have that for sufficiently small $\lambda$ the expectation $\mathbb{E}_\rho(\exp(C\lambda|\psi_{k - j}|^2|k - j|^2))$ is uniformly bounded in $k, j$.  Thus the expectation \eqref{ExpectationToBound} is uniformly bounded in $k$.

Next we estimate differences of the truncated operator.  We can write using indicator functions that
\begin{equation}
H^N(\psi)^\pm_k = \frac{1}{2|k|i^\pm} \sum_{j_1 + j_2 = k} \mathbf{1}_{|k|, |j_1|, |j_2| \leq N} \frac{(j_1^\perp \cdot j_2)(|j_2| - |j_1|)}{|j_1|^{2 + 2\delta} + |j_2|^{2 + 2\delta}} (\psi_{j_1} \psi_{j_2} \pm \psi_{-j_1} \psi_{-j_2}).
\end{equation}
Using the notation $\gamma_{k, j}$ introduced in \eqref{EOfDerivOfH}, we can estimate for any $M \leq N$ the difference
\begin{equation}\label{DifferenceOfH}
\mathcal{E}^\delta(H^N(\psi)_k - H^M(\psi)_k) \precsim \sum_{j \in \mathbb{Z}^2} (\mathbf{1}_{|k|, |j|, |k - j| \leq N} - \mathbf{1}_{|k|, |j|, |k - j| \leq M})^2 \gamma_{k, j} |\psi_{k - j}|^2 |k - j|^2.
\end{equation}
We coarsely estimate the difference of the indicator functions above by writing
\begin{align*}
\mathbf{1}_{|k|, |j|, |k - j| \leq N} - \mathbf{1}_{|k|, |j|, |k - j| \leq M} & \leq \mathbf{1}_{|k|, |j|, |k - j| \leq N} \mathbf{1}_{\{|k| > M \text{ or } |j| > M \text{ or } |k - j| > M\}} \\
& \leq \mathbf{1}_{|k|, |j|, |k - j| \leq N} (\mathbf{1}_{|k| > M} + \mathbf{1}_{|j| > M} + \mathbf{1}_{|k - j| > M}) \\
& \leq \mathbf{1}_{|k| > M} + \mathbf{1}_{|j| > M} + \mathbf{1}_{|k - j| > M}.
\end{align*}
In preparation for controlling the expectation of \eqref{DifferenceOfH} as did for \eqref{ExpectationToBound}, we estimate the quantity
\begin{align*}
& \quad \sum_{j \in \mathbb{Z}^2} (\mathbf{1}_{|k|, |j|, |k - j| \leq N} - \mathbf{1}_{|k|, |j|, |k - j| \leq M})^2 \gamma_{k, j} \\
& \precsim \mathbf{1}_{|k| > M} \int_{\mathbb{R}^2} \frac{|x|^{2 + 2\delta}}{(|x|^{2 + 2\delta} + |k - x|^{2 + 2\delta})^2} \, dx + \int_{|x| > M} \frac{|x|^{2 + 2\delta}}{(|x|^{2 + 2\delta} + |k - x|^{2 + 2\delta})^2} \, dx \\
& \leq |k|^{-2\delta}\mathbf{1}_{|k| > M} + \int_{|x| > M} \frac{dx}{|x|^{2 + 2\delta}} \\
& \precsim_\delta M^{-2\delta}.
\end{align*}
This gives us the estimate
\begin{align*}
& \quad \mathbb{E}_\rho \exp(\lambda M^{2\delta} \mathcal{E}^\delta(H^N(\psi)_k - H^M(\psi)_k)) \\
& \leq \frac{\sum_{j \in \mathbb{Z}^2} (\mathbf{1}_{|k|, |j|, |k - j| \leq N} - \mathbf{1}_{|k|, |j|, |k - j| \leq M})^2 \gamma_{k, j} \mathbb{E}_\rho(\exp(C\lambda|\psi_{k - j}|^2|k - j|^2))}{\sum_{j \in \mathbb{Z}^2} (\mathbf{1}_{|k|, |j|, |k - j| \leq N} - \mathbf{1}_{|k|, |j|, |k - j| \leq M})^2 \gamma_{k, j}},
\end{align*}
and so for sufficiently small $\lambda$ this quantity is uniformly bounded in $k$, $N$, and $M$.
\end{proof}

We can use these exponential estimates to derive further bounds on small-time differences of the time integral of the nonlinearity.  These will permit us to show the compactness needed to extract our candidate solution.

\begin{proposition}\label{GBounds}

Let $p \in \mathbb{N}$ be given.  Denote $G_M(t)_k := \int_0^t B_M(\psi(\tau))_k \, dt$ and $G_{N, M}(t)_k = G_N(t)_k - G_M(t)_k$ for $N > M$.  Then we have the following estimates:
\begin{itemize}
\item[(a)]{$\|\sup_{t\in [0, T]} G_M(t)_k\|_{L^{2p}(\rho)} \precsim MT$.}
\item[(b)]{$\|\sup_{t \in [0, T]} |G_M(t)_k| \|_{L^{2p}(\rho)} \precsim |k|^{-\delta}T^\frac12$.}
\item[(c)]{For every $N > M$, $\|\sup_{t\in[0, T]} |G_{N, M}(t)_k|\|_{L^{2p}(\rho)} \precsim |k|^{-\delta}T^\frac12$.}
\item[(d)]{For all $t_1, t_2$ so that $|t_2 - t_1|$ is sufficiently small depending on $N$ and $\delta$, we have $\sup_{N \geq 0} \|\sup_{t \in [0, T]} |G_N(t_2)_k - G_N(t_1)_k|\|_{L^{2p}(\rho)} \precsim |t_2 - t_1|^{\frac{1 + 2\delta}{2 + 2\delta}}$.}
\end{itemize}
\end{proposition}

\begin{proof}
To prove (a), notice that we first have using \eqref{NonlinearityComponent} the coarse bound

\begin{align*}
& \quad \|B_M(\psi)_k\|_{L^{2p}(\rho)}^{2p} \\
& \leq \left(\prod_{j = 1}^p \frac{\mathbf{1}_{|h_{1, j}|, |h_{2, j}| \geq M}}{|k|} \sum_{h_{1, j} + h_{2, j} = k} (h_{i, j} \cdot h_{2, j}^\perp)(|h_{1, j}| - |h_{2, j}|)\right)\int \left(\prod_{j = 1}^p |\psi_{h_{1, j}} \psi_{h_{2, j}}|\right) d\rho(\psi) \\
& \precsim_p \left(\prod_{j = 1}^p \sum_{h_{1, j} + h_{2, j} = k} \mathbf{1}_{|h_{1, j}|, |h_{2, j}| \geq M}(h_{i, j} \cdot h_{2, j}^\perp)\frac{|h_{1, j}| - |h_{2, j}|}{|k|}\right)\left(\prod_{j = 1}^p \frac{1}{|h_{1, j}| |h_{2, j}|}\right) \\
& \precsim \prod_{j = 1}^p \sum_{h_{1, j} + h_{2, j} = k} \mathbf{1}_{|h_{1, j}|, |h_{2, j}| \geq M}\left(\frac{h_{1, j}}{|h_{1, j}|} \cdot \frac{h_{2, j}^\perp}{|h_{2, j}|}\right)\frac{|h_{1, j}| - |h_{2, j}|}{|k|} \\
& \precsim \prod_{j = 1}^p \sum_{h_{1, j} + h_{2, j} = k} \mathbf{1}_{|h_{1, j}|, |h_{2, j}| \geq M} \\
& \precsim_p M^{2p}
\end{align*}
from which, if we denote $G_M(t) := \int_0^t B_M(\psi(\tau)) \, dt$ we have whenever $0 \leq t \leq T$ that
\begin{equation}
\|\sup_{t \in [0, T]} |G_M(t)_k|\|_{L^{2p}(\rho)} \leq \int_{0}^{T} \|B_M(\psi(\tau))_k\|_{L^{2p}(\rho)}\,d\tau \leq MT.
\end{equation}

We can also estimate this term in a different way in (b) using the fact that $L_0 H_M = B_M$ along with the Burkholder-Davis-Gundy inequality as follows:

\begin{align*}
\left\|\sup_{t \in [0, T]} |G_M(\psi(t))_k|\right\|_{L^{2p}(\rho)} & = \left\|\sup_{t \in [0, T]} \int_{0}^{t} L_0H_M(\psi(\tau)_k)\,d\tau\right\|_{L^{2p}(\rho)} \\
& \precsim_p \left\| \int_{0}^{T} \mathcal{E}^\delta(H_M(\psi(\tau))_k)\,d\tau\right\|_{L^{p}(\rho)}^\frac12 \\
& \precsim_p T^\frac12 \left(\left\|\mathcal{E}^\delta(H_M(\psi(\tau))_k) \right\|_{L^{p}(\rho)}\right)^\frac12.
\end{align*}
Now we may use the elementary inequality $|f|^p \precsim_p e^f$ and Proposition \ref{ExponentialBounds} (and reintroduce the small quantity $\lambda$ depending on $\delta$ appearing in that proposition) to further estimate that
\begin{align*}
\left\|\sup_{t \in [0, T]} |G_M(\psi(t))_k|\right\|_{L^{2p}(\rho)} & \precsim_{p, \delta} T^\frac12|k|^{-\delta} \left(\int e^{\lambda|k|^{2\delta}\mathcal{E}^\delta(H_M(\psi(\tau))_k)} d\rho(\psi)\right)^\frac{1}{2p} \\
& \leq |k|^{-\delta}T^\frac12.
\end{align*}
Applying the exact same argument using the difference estimate yields inequality (c).  To show (d), we interpolate between (a) and (c) by writing, $G_{N, M}(t)_k := G_N(t)_k - G_M(t)_k$ with $M < |k| < N$,
\begin{align*}
\|G_N(t_2)_k - G_N(t_1)_k\|_{L^{2p}(\rho)} & \leq \|G_M(t_2)_k - G_M(t_1)_k \|_{L^{2p}(\rho)} + \|G_{N, M}(t_2)_k - G_{N, M}(t_1)_k \|_{L^{2p}(\rho)} \\
& \precsim M|t_2 - t_1| + M^{-\delta}|t_2 - t_1|^\frac12,
\end{align*}
and choosing $M \sim |t_2 - t_1|^{-\frac{1}{2(1 + \delta)}}$ to balance the two terms yields (d).
\end{proof}

We can perform almost exactly the same estimates for the semigroup convolution of the drift term.  Denote
\begin{equation}
\tilde{G}_M(t) = \int_0^t e^{-|D|^{2\delta}(t - \tau)} B_M(\psi(\tau)) \, d\tau.
\end{equation}
Then we have
\begin{proposition}\label{FinerEstimatesSemigroup}
Let $p \in \mathbb{N}$ be given.  Then we have the following estimates:
\begin{itemize}
\item[(a)]{$\|\sup_{0 \leq t \leq T} \tilde{G}_M(t)_k\|_{L^{2p}(\rho)} \precsim M\left(\frac{1 - e^{-|k|^{-2\delta}T}}{2|k|^{2\delta}}\right)$.}
\item[(b)]{$\|\sup _{0 \leq t \leq T} \tilde{G}_M(t)_k \|_{L^{2p}(\rho)} \precsim |k|^{-\delta}\left(\frac{1 - e^{-|k|^{-2\delta}T}}{2|k|^{2\delta}}\right)^\frac12$.}
\item[(c)]{For every $N > M$, $\|\sup _{0 \leq t \leq T} \tilde{G}_{N, M}(t)_k \|_{L^{2p}(\rho)} \precsim |k|^{-\delta}\left(\frac{1 - e^{-|k|^{-2\delta}T}}{2|k|^{2\delta}}\right)^\frac12$.}
\item[(d)]{For any $\epsilon \in (0, 1)$ and $|t_2 - t_1| \leq 1$, we have $$\| \tilde{G}_N(t_2)_k - \tilde{G}_N(t_1)_k \|_{L^{2p}(\rho)} \precsim |k|^{-2\delta + 2\delta\epsilon}|t_2 - t_1|^\epsilon.$$}
\end{itemize}
\end{proposition}

\begin{proof}
Parts (a), (b), (c) are a line by line reproof of the previous proposition, except that instead of using $\sup_{0 \leq t \leq T} \int_0^t \, d\tau = T$ to coarsely estimate the time integrals, we use the bound $$\int_0^t e^{-|k|^{2\delta}(t - \tau)} \, d\tau = \frac{1 - e^{-|k|^{-2\delta}t}}{2|k|^{2\delta}}.$$

To show part (d), we consider time differences in the semigroup convolution with the drift.  Without loss of generality we consider times $t_1, t_2 \in [0, T]$ satisfying $|t_2 - t_1| \leq 1$.  Then, by following the proof of (b) in the previous proposition and by using (b) above, we can estimate as follows:
\begin{align*}
& \quad \| \tilde{G}_N(t_2)_k - \tilde{G}_N(t_1)_k \|_{L^{2p}(\rho)} \\
& \leq \left\|\int_{t_1}^{t_2} e^{-|k|^{2\delta}(t_2 - \tau)} B_N(\psi_\tau, \psi_\tau) \, d\tau \right\|_{L^{2p}(\rho)} + (e^{-|k|^{2\delta}|t_2 - t_1|} - 1)\left\|\tilde{G}_N(t_1)_k \right\|_{L^{2p}(\rho)} \\
& \precsim |k|^{-\delta}|t_2 - t_1|^\frac12 + (e^{-|k|^{2\delta}|t_2 - t_1|} - 1)|k|^{-2\delta} \\
& \precsim |t_2 - t_1|^\frac12.
\end{align*}
However, we can also coarsely estimate each term in the difference by (b), yielding the bound
\begin{equation}
\| \tilde{G}_N(t_2)_k - \tilde{G}_N(t_1)_k \|_{L^{2p}(\rho)} \precsim |k|^{-2\delta}.
\end{equation}
Interpolating between these two bounds then gives the desired result.
\end{proof}

We pause to note that these propositions give us sufficient control of the truncated drifts to construct our candidate for the drift of our energy solution.

\begin{proposition}\label{ConstructDrift}
Suppose that $\psi \in \mathscr{R}_\delta(0, T)$ is given, and let $\epsilon > 0$ be given.  Then there exists a process which we denote by $\int_0^t B(\psi_\tau, \psi_\tau) \, d\tau$ for which
$$\lim_{N \to \infty} \int_0^t B_N(\psi^{(N)}_\tau) \, d\tau = \int_0^t B(\psi_\tau) \, d\tau,$$
with the convergence being in the sense of $L^{2p}_\rho C([0, T], \mathcal{F}L^{2p, 0})$ for $p > \delta^{-1}$.  In particular the convergence is in the sense of $C([0, T], \mathcal{F}L^{\infty, 0})$ almost surely.
\end{proposition}

\begin{proof}
Note that by Proposition \ref{GBounds} (c) with $t_2 = T$ and $t_1 = 0$ we have when $M < N$ that
$$\left\| \sup_{0 \leq t \leq T} \left|\int_0^t B_N(\psi_\tau)_k d\tau - \int_0^t B_M(\psi_\tau)_k d\tau \right|\right\|_{L^{2p}(\rho)} \precsim |k|^{-\delta}T^\frac12.$$
Then we can calculate that there is a random time $t_* = t_*(N, M, \epsilon) \in [0, T]$ for which we can estimate whenever $p\delta > 1$ that
\begin{align*}
& \quad\;\left\|\,\sup_{t \in [0, T]} \left\| \int_0^t B_N(\psi^{(N)}_\tau)_k d\tau - \int_0^t B_M(\psi^{(M)}_\tau)_k d\tau \right\|_{\ell^{2p}} \right\|_{L^{2p}(\rho)}^{2p} \\
& \leq \left\|\,\left\| \int_0^{t_*} B_N(\psi^{(N)}_\tau)_k d\tau - \int_0^{t_*} B_M(\psi^{(M)}_\tau)_k d\tau \right\|_{\ell^{2p}} + \epsilon \right\|_{L^{2p}(\rho)}^{2p} \\
& \leq 2^{2p}\left(\left\|\,\left\| \int_0^{t_*} B_N(\psi^{(N)}_\tau)_k d\tau - \int_0^{t_*} B_M(\psi^{(M)}_\tau)_k d\tau \right\|_{\ell^{2p}} \right\|_{L^{2p}(\rho)}^{2p} + \epsilon^{2p}\right) \\
& = 2^{2p}\left(\left\|\,\left\| \int_0^{t_*} B_N(\psi^{(N)}_\tau)_k d\tau - \int_0^{t_*} B_M(\psi^{(M)}_\tau)_k d\tau \right\|_{L^{2p}(\rho)} \right\|_{\ell^{2p}}^{2p} + \epsilon^{2p} \right)\\
& \leq 2^{2p}T^\frac12 \sum_{|k| \geq M} |k|^{-2p\delta} + 2^{2p}\epsilon^{2p} \\
& \precsim 2^{2p} T^\frac12 \frac{M^{-2p\delta + 2}}{2p\delta - 2} + 2^{2p}\epsilon^{2p},
\end{align*}
from which we have
\begin{align*}
\|G_N(\cdot) - G_M(\cdot)\|_{L^{2p}(\rho)C^0_t\mathcal{F}L^{2p, 0}} & \leq 2T^\frac{1}{4p} M^{-\delta + \frac{1}{p}} + 2\epsilon \\
& \leq 2(1 + T)^\frac{1}{8\delta} M^{-\frac{\delta}{2}} + 2\epsilon \\
& \leq 3\epsilon,
\end{align*}
provided we choose $M$ sufficiently large depending on $\delta$, $T$, and $\epsilon$.
\end{proof}

\section{Extraction of a Candidate Solution}

From our sequence of approximate solutions we would like to extract a candidate solution using compactness.  We will need to use the fact that the approximate solutions $\psi^{(N)}$ satisfy the following two Duhamel integral equations:
\begin{align}\label{DuHamelFormulation}
\psi^{(N)}_t & = \psi^{(N)}_0 + \int_0^t -|D|^{2\delta}\psi^{(N)}_\tau + |D|^\delta W_\tau \notag \\
& = e^{-|D|^{2\delta}t}\psi_0 + \int_0^t e^{-|D|^{2\delta}(t - \tau)} B_N(\psi^{(N)}_\tau, \psi^{(N)}_\tau)\, dt + \int_0^t e^{-|D|^{2\delta}(t - \tau)} |D|^{\delta} dW_\tau.
\end{align}
Our first step is to establish tightness of the sequence of approximate solutions.

\begin{proposition}
Let $\delta > 0$.  Denote $$\mathcal{A}^{(N)}_t := G_N(t) = \int_0^t B_N(\psi^{(N)}_\tau, \psi^{(N)}_\tau) \, dt$$ and $$\tilde{\mathcal{A}}^{(N)}_t := \tilde{G}_N(t) = \int_0^t e^{-|D|^{2\delta}(t - \tau)} B_N(\psi^{(N)}_\tau, \psi^{(N)}_\tau)\, dt.$$ Suppose that the initial data $\psi_0 \in \mathcal{F}L^{\infty, 2\delta - 3\delta\epsilon}$.  Then for all sufficiently small $\epsilon$ depending on $\delta$, the sequence $(\psi^{(N)}, \mathcal{A}^{(N)}, \tilde{\mathcal{A}}^{(N)})$ is tight in $C([0, T] : \mathfrak{X})$, where $$\mathfrak{X} = \mathcal{F}L^{\infty, (1 - \epsilon) \wedge (2\delta - 3\delta\epsilon)} \times \mathcal{F}L^{\infty, \delta - 2\delta\epsilon} \times \mathcal{F}L^{\infty, 2\delta - 3\delta\epsilon}.$$
\end{proposition}

\begin{proof}
By Proposition \ref{FinerEstimatesSemigroup}(d), we have
\begin{align*}
\|\tilde{\mathcal{A}}^{(N)}_t - \tilde{\mathcal{A}}^{(N)}_{t^\prime}\|_{L^{2p}_\rho \mathcal{F}L^{2p, 2\delta - 3\delta\epsilon}}^{2p} & = \mathbb{E}_\rho\left(\sum_{k \in \mathbb{Z}^2_0} |k|^{(2\delta - 3\delta\epsilon)p}|\tilde{\mathcal{A}}^{(N)}_t - \tilde{\mathcal{A}}^{(N)}_{t^\prime}|^p\right) \\
& \leq |t - t^\prime|^{\epsilon p} \sum_{k \in \mathbb{Z}^2_0} |k|^{-\delta\epsilon p} \\
& \leq C|t - t^\prime|^{\epsilon p},
\end{align*}
if we choose $p$ sufficiently large depending on $\delta$ and $\epsilon$. By the classical Arzela-Ascoli compactness criterion in abstract Weiner spaces, this demonstrates tightness of $(\tilde{\mathcal{A}}^{(N)})$ in $C([0, T] : \mathcal{F}L^{2p, 2\delta - 3\delta\epsilon})$ for arbitrarily large $p$, and therefore also in $C([0, T] : \mathcal{F}L^{\infty, 2\delta - 3\delta\epsilon})$.  Similarly, Proposition \ref{GBounds}(d) yields the estimate $\|\mathcal{A}^{(N)}_t - \mathcal{A}^{(N)}_{t^\prime}\|_{L^{2p}_\rho \mathcal{F}L^{2p, \delta - 2\delta\epsilon}}^{2p} \leq |t - t^\prime|^\frac12$ whenever $|t - t^\prime|$ is sufficiently small, from which we find that $(\mathcal{A}^{(N)})$ is tight in $C([0, T] : \mathcal{F}L^{\infty, \delta - 2\delta\epsilon})$.

By applying the Kolmogorov continuity criterion (c.f. \cite{Kuo}), a routine calculation shows that the Wiener integral $$\int_0^t e^{-|D|^{2\delta}(t - \tau)} |D|^{\delta - 1} dW_\tau$$ is in $C([0, T] : \mathcal{F}L^{\infty, 1 - \epsilon})$ for some $\epsilon > 0$ sufficiently small.  But now applying all of these estimates to \eqref{DuHamelFormulation} implies that $(\psi^{(N)})$ is tight in 
$C([0, T] : \mathcal{F}L^{\infty, 2\delta - 3\delta\epsilon})$ provided we take $1 > 2\delta$ and $\epsilon > 0$ sufficiently small.

\end{proof}

Now, using the Prokhorov and Skorohod Lemmas, there are up to a change in the underlying probability space stochastic processes $(\psi, \mathcal{A}, \tilde{\mathcal{A}})$ with values in $C([0, T] : \mathfrak{X})$ so that upon passing to a subsequence, $(\psi^{(N)}, \mathcal{A}^{(N)}, \tilde{\mathcal{A}}^{(N)}) \to (\psi, \mathcal{A}, \tilde{\mathcal{A}})$ almost surely.  To finish our proof of existence of solutions, we need only verify that this candidate satisfies the definition of energy solution.

\begin{proposition}\label{Existence}
Let $\delta > 0$ and $T > 0$ be given.  There exists a Wiener process $W$ on $L^2$ for which the processes $(\psi, \mathcal{A}, \tilde{\mathcal{A}})$ satisfy the relations
\begin{equation}
\psi_t = \psi_0 + \mathcal{A}_t - \int_0^t |D|^{2\delta} \psi_\tau \, d\tau + |D|^{\delta - 1} W_t 
\end{equation}
in the sense of $C([0, T] : \mathcal{F}L^{\infty, -3\epsilon\delta})$, as well as
\begin{equation}
\psi_t = e^{-|D|^{2\delta}t}\psi_0 + \tilde{\mathcal{A}}_t + \int_0^t e^{-|D|^{2\delta}(t - \tau)} |D|^{\delta - 1} dW_\tau 
\end{equation}
in the sense of $C([0, T] : \mathcal{F}L^{\infty, (1 - \epsilon) \wedge (2\delta - 3\epsilon\delta)})$, and moreover we have in the distributional sense that
\begin{equation}
\tilde{\mathcal{A}}_t = \int_0^t e^{-|D|^{2\delta}(t - \tau)} d\mathcal{A}_\tau.
\end{equation}
\end{proposition}

\begin{proof}
We take as our candidate solution the triple $(\psi, \mathcal{A}, \tilde{\mathcal{A}})$ constructed above as the limit of the subsequence $(\psi^{(N)}, \mathcal{A}^{(N)}, \tilde{\mathcal{A}}^{(N)})$.  We have already shown by construction that $\psi^{(N)} \to \psi$ in $C([0, T] : \mathcal{F}L^{\infty, 2\delta - 3\delta\epsilon})$.  We next consider the contribution from the Duhamel drift term
\begin{align*}
& \quad \int_0^t B_N(\psi^{(N)}_\tau, \psi^{(N)}_\tau)\, dt - \int_0^t B_M(\psi_\tau, \psi_\tau)\, dt \\
& = \int_0^t B_N(\psi^{(N)}_\tau, \psi^{(N)}_\tau) - B_M(\psi^{(N)}_\tau, \psi^{(N)}_\tau) \, dt \\
& \; + \int_0^t B_M(\psi^{(N)}_\tau, \psi^{(N)}_\tau) - B_M(\psi_\tau, \psi_\tau) \, dt \\
& := D_1 + D_2.
\end{align*}
We take $\lim_{M \to \infty} \limsup_{N \to 
\infty}$ and may assume without loss of generality that $M < N$.  Now since the functional $B_M$ in $D_2$ depends on finitely many frequency components, it is continuous as a mapping $\mathcal{F}L^{\infty, 2\delta - 3\delta\epsilon} \to \mathcal{F}L^{\infty, 2\delta - 3\delta\epsilon}$.  Hence we have that $\|D_2\|_{\mathcal{F}L^{\infty, 2\delta - 3\delta\epsilon}} \to 0$ as $N \to \infty$, since $\psi^{(N)} \to \psi$.  Next, estimating as in Proposition \ref{ConstructDrift}, we have
\begin{align*}
& \quad \lim_{M \to \infty} \limsup_{N \to \infty} \left\| \left\|\int_0^t e^{-|D|^{2\delta}(t - \tau)} \left(B_N(\psi^{(N)}_\tau, \psi^{(N)}_\tau)_k - B_M(\psi^{(N)}_\tau, \psi^{(N)}_\tau)_k\right) \, d\tau \right\|_{C\mathcal{F}L^{\infty, 2\delta - 3\delta\epsilon}} \right\|_{L^p(\mathbb{P})} = 0
\end{align*}
so that, after possibly passing to another subsequence, $D_1 \to 0$ almost surely in $C\mathcal{F}L^{\infty, 2\delta - 3\delta\epsilon}  $.  We observe that by construction $\mathcal{A}_0 = 0$ and by Proposition \ref{GBounds}(d), $\mathcal{A}_t$ has paths of H\"older class $\frac{1 + 2\delta}{2 + 2\delta} > \frac12$ in time when $\delta > 0$, and so $\mathcal{A}_t$ has zero quadratic variation.  Finally, we note that since each $\mathcal{A}^{(N)}_\tau$ converges to $\mathcal{A}_\tau$ in $C^{\frac12^-}([0, T] : \mathcal{F}L^{\infty, 2\delta - 3\delta\epsilon})$, we have for each $k \in \mathbb{Z}_0^2$ that
\begin{align*}
\left\langle \int_0^t e^{-|D|^{2\delta}(t - \tau)} d\mathcal{A}_\tau^{(N)}, e_k \right\rangle \to \left\langle \int_0^t e^{-|D|^{2\delta}(t - \tau)} d\mathcal{A}_\tau, e_k \right\rangle,
\end{align*}
where the above integrals are taken to be in the sense of Young.
\end{proof}

\section{Pathwise Uniqueness in the Highly Diffusive Case}

We record here a pathwise uniqueness result for \eqref{SQGstreamline} with sufficiently strong diffusion, following the argument of \cite{gj}.  We take uniqueness in the following sense.

\begin{definition}\label{PathwiseUniquenessDefinition}
(Pathwise Uniqueness)  Solutions to \eqref{SQGstreamline} are pathwise unique in the class $X$ when, given two energy solutions $\psi, \tilde{\psi}$ defined on a common probability space $(\Omega, \mathfrak{F}, \mathbb{P})$ that generate the same Brownian motion $W$, for which $\psi(0) = \tilde{\psi}(0)$ almost surely, the set $$\left\{ \sup_{t > 0} \|\psi_t - \tilde{\psi}_t\|_{X} > 0\right\}$$ is of $\mathbb{P}$-measure zero.
\end{definition}

Once pathwise uniqueness is established for mean-zero solutions to \eqref{SQGstreamline}, it follows also for mean zero solutions of \eqref{SQGTheta}.  Our uniqueness result is then given in

\begin{proposition}
For $\delta > \frac32$ fixed, let $\epsilon > 0$ be chosen sufficiently small so that the conclusion of Proposition \ref{FinerEstimatesSemigroup} holds.  Then energy solutions (in the sense of Definitions \ref{ControlledProcessDefinition}-\ref{EnergySolutionDefinition}) to \eqref{SQGstreamline} are pathwise unique in the class $\mathcal{F}L^{\infty, 2\delta - 3\delta\epsilon}$.  
\end{proposition}

Let $\psi, \tilde{\psi}$ be two energy solutions of \eqref{SQGstreamline} as in Definition \ref{PathwiseUniquenessDefinition} with $X = \mathcal{F}L^{\infty, (-\epsilon) \wedge (2\delta - 3\delta\epsilon)}$.  Set $s = 2\delta - 3\delta\epsilon$.  Now $\tilde{\psi}$ was constructed using a sequence of Galerkin approximations $\psi^{(N)}$.  We first claim that there exists a $T > 0$ depending on $p, \epsilon, \psi_0$ so that $\Pi_N(\psi_t - \psi_t^{(N)}) \to 0$ in the sense of $C([0, T] : \mathcal{F}L^{\infty, s})$ almost surely.  Once this is established, we will show pathwise uniqueness holds on arbitrarily long time intervals by a continuity argument.

\textbf{Step 1 : Pathwise uniqueness on some positive time interval.}  Denote $D_t^{(N)} = \Pi_N(\psi_t - \psi_t^{(N)})$ and $\mathcal{D}_N = \|D_t^{(N)}\|_{L^\infty_t \mathcal{F}L^{\infty, s}_x}$.  It now suffices to show that $\mathcal{D}_N \to 0$ as $N \to \infty$ for some $T > 0$.  The difference equation for $D_t^{(N)}$, in which the common Brownian motion $W$ and initial data contributions have cancelled almost surely, reads
\begin{align*}
D_t^{(N)} & = \int_0^t e^{-|D|^{2\delta}(t - \tau)}\Pi_N\left(B(\psi_\tau) - B_N(\psi_\tau^{(N)}\right) \, d\tau \\
& = \int_0^t e^{-|D|^{2\delta}(t - \tau)}\Pi_N\left(B(\psi_\tau) - B(\Pi_N \psi_\tau)\right) \, d\tau + \int_0^t e^{-|D|^{2\delta}(t - \tau)}\Pi_N\left(B(\Pi_N \psi_\tau) - B_N(\psi_\tau^{(N)})\right) \, d\tau \\
& = \int_0^t e^{-|D|^{2\delta}(t - \tau)}\Pi_N\left(B(\Pi_N \psi_\tau, \Pi_N \psi_\tau) - B_N(\psi_\tau^{(N)}, \psi_\tau^{(N)})\right) \, d\tau \\
& \quad + \Pi_N \int_0^t e^{-|D|^{2\delta}(t - \tau)}\left(B(\psi_\tau) - B_N(\psi_\tau)\right) \, d\tau \\
& := \int_0^t e^{-|D|^{2\delta}(t - \tau)}\left(B_N(\Pi_N \psi_\tau, D_\tau^{(N)}) + B_N(D_\tau^{(N)}, \psi_\tau^{(N)})\right) \, d\tau \\
& \quad + \Pi_N \varphi_t^{(N)},
\end{align*}
where here we have introduced
\begin{equation}
\varphi_t^{(N)} := \int_0^t e^{-|D|^{2\delta}(t - \tau)}\left(B(\psi_\tau) - B_N(\psi_\tau)\right) \, d\tau.
\end{equation}
By Propositions \ref{FinerEstimatesSemigroup}, (b) and (c), we can for sufficiently large $p$ bound the assumed solution $\varphi^{(N)}$ by
\begin{equation*}
\|\varphi_t^{(N)}(e_k)\|_{L_\mathbb{P}^{2p} L_t^{2p}} \leq C_p\max\left(|k|^{-2\delta}, |k|^{-\delta}N^{-\delta}\right),
\end{equation*}
from which we may interpolate to conclude that for any $0 \leq \theta \leq 1$ and sufficiently large $p$,
\begin{equation*}
\|\varphi_t^{(N)}(e_k)\|_{L_\mathbb{P}^{2p} L_t^{2p}} \leq C_p|k|^{-\delta(2 - \theta)}N^{-\delta\theta}.
\end{equation*}
If we set
\begin{equation*}
\Phi_N = \sup_{k \in \mathbb{Z}^2_0} |k|^{2\delta - 3\delta\epsilon} \|\varphi_t^{(N)}(e_k)\|_{L_t^{\infty}},
\end{equation*}
then we may bound $\Phi_N$ almost surely as follows: for sufficiently large $p$ we have
\begin{align*}
\mathbb{E}_\mathbb{P}\left(\sum_{N = 1}^\infty N\Phi_N^{2p}\right) & = \sum_{N = 1}^\infty N\mathbb{E}_\mathbb{P}(\Phi_N^{2p}) \\
& \leq \sum_{N = 1}^\infty N\mathbb{E}_\mathbb{P}\left(\sup_{k \in \mathbb{Z}^2_0} |k|^{2\delta - 4\delta\epsilon} \|\varphi_t^{(N)}(e_k)\|_{L_t^{\infty}}^p \right) \\
& = \sum_{N = 1}^\infty N\left\|\sup_{k \in \mathbb{Z}^2_0} |k|^{2\delta - 3\delta\epsilon} \|\varphi_t^{(N)}(e_k)\|_{L_t^{\infty}} \right\|_{L^{2p}_\mathbb{P}}^p \\
& \leq \sum_{N = 1}^\infty N \sup_{k \in \mathbb{Z}^2_0} \left\||k|^{2\delta - 3\delta\epsilon} \|\varphi_t^{(N)}(e_k)\|_{L_t^{\infty}} \right\|_{L^{2p}_\mathbb{P}}^p \\
& \leq C_p\sum_{N = 1}^\infty N^{-3\delta\epsilon p}
\end{align*}
which converges provided we choose $p$ sufficiently large depending on $\epsilon$ and $\delta$.  Thus almost surely we have $\Phi_N \leq C_p N^{-\frac{1}{p}}$.  The other term can be estimated as follows:
\begin{align*}
& \sup_{t \in [0, T]} \left|\left(\int_0^t e^{-|D|^{2\delta}(t - \tau)}\left(B_N(\Pi_N \psi_\tau, D_\tau^{(N)}) + B_N(D_\tau^{(N)}, \psi_\tau^{(N)})\right) \, d\tau\right)(e_k)\right| \\
& \leq \sup_{t \in [0, T]} \int_0^t e^{-|k|^{2\delta}(t - \tau)} \sum_{h \neq 0, k} |\alpha_{h, k - h, k}|\,\left(|\Pi_N\psi_\tau(e_h)|\,|D_\tau^{(N)}(e_{k - h})| + |D_\tau^{(N)}(e_h)|\,|\psi_\tau^{(N)}(e_{k - h})|\right)\, d\tau \\
& \leq \sup_{t \in [0, T]} \int_0^t e^{-|k|^{2\delta}(t - \tau)} \sum_{h \neq 0, k} |k - h|\,|h|\,\left(|\Pi_N\psi_\tau(e_h)|\,|D_\tau^{(N)}(e_{k - h})| + |D_\tau^{(N)}(e_h)|\,|\psi_\tau^{(N)}(e_{k - h})|\right)\, d\tau \\
& \leq \sup_{t \in [0, T]} \int_0^t e^{-|k|^{2\delta}(t - \tau)} \left(\sum_{h \neq 0, k} |k - h|\, |h|\,\left(|\Pi_N\psi_\tau(e_h)|\,|k - h|^{-s} + |h|^{-s}\,|\psi_\tau^{(N)}(e_{k - h})|\right)\right) \mathcal{D}_N\, d\tau \\
& := \mathcal{D}_N|k|^{-s}\mathcal{I}_N(k, T),
\end{align*}
where we have introduced
\begin{equation*}
\mathcal{I}_N(k, T) := \sup_{k \in \mathbb{Z}^2_0} |k|^{s} \sup_{t \in [0, T]} \int_0^t e^{-|k|^{2\delta}(t - \tau)} \left(\sum_{h \neq 0, k} |h|\, |k - h|\left(|\Pi_N\psi_\tau(e_h)|\,|k - h|^{-s} + |h|^{-s}\,|\psi_\tau^{(N)}(e_{k - h})|\right)\right)\, d\tau.
\end{equation*}
But then it follows that we have for all $N$ that
\begin{equation}\label{DiffInequality}
\mathcal{D}_N \leq \mathcal{I}_N(k, T)\mathcal{D}_N + \Phi_N.
\end{equation}
We continue to estimate $\mathcal{I}_N$; if we choose $s > 3$, then using H\"older's inequality with respect to $|k - h|^{1 - s} dt \, dh$ and $|h|^{1 - s} dt \, dh$ respectively yields
\begin{align*}
\mathcal{I}_N & \leq \sup_{k \in \mathbb{Z}^2_0} C_{p, \epsilon} |k|^{s} \sup_{t \in [0, T]} \left(\int_0^t e^{-p^\prime|k|^{2\delta}(t - \tau)} \, d\tau\right)^{\frac{1}{p^\prime}} \\
& \qquad \qquad \times \left(\int_0^t \sum_{h \neq 0, k} \frac{|h|^p \,|\Pi_N\psi_\tau(e_h)|^p}{|k - h|^{s - 1}} + \sum_{h \neq 0, k} \frac{|k - h|^p\,|\psi_\tau^{(N)}(e_{k - h})|^p}{|h|^{s - 1}} \, d\tau\right)^\frac{1}{p}\\
& \leq \sup_{k \in \mathbb{Z}^2_0} C_{p, \epsilon}|k|^{s - 2\delta/p^\prime} \sup_{t \in [0, T]} \left(\int_0^t \sum_{h \neq 0, k} \frac{|h|^p \,|\Pi_N\psi_\tau(e_h)|^p}{|k - h|^{s - 1}} + \sum_{h \neq 0, k} \frac{|k - h|^p\,|\psi_\tau^{(N)}(e_{k - h})|^p}{|h|^{s - 1}} \, d\tau\right)^\frac{1}{p} \\
& \leq \sup_{k \in \mathbb{Z}^2_0} C_{p, \epsilon} \sup_{t \in [0, T]} \left(\int_0^t \sum_{h \neq 0, k} \frac{|h|^p \,|\Pi_N\psi_\tau(e_h)|^p}{|k - h|^{s - 1}} + \sum_{h \neq 0, k} \frac{|k - h|^p\,|\psi_\tau^{(N)}(e_{k - h})|^p}{|h|^{s - 1}} \, d\tau\right)^\frac{1}{p},
\end{align*}
where we have chosen $p$ sufficiently large so that $s - 2\delta/p^\prime \leq 0$.  Now by hypercontractivity, we find that
\begin{align*}
\mathbb{E}[\mathcal{I}_N^p] & \leq \sup_{k \in \mathbb{Z}^2_0} C_{p, \epsilon} \sup_{t \in [0, T]} \int_0^t \sum_{h \neq 0, k} \frac{|h|^p \,\mathbb{E}[|\Pi_N\psi_\tau(e_h)|^p]}{|k - h|^{s - 1}} + \sum_{h \neq 0, k} \frac{|k - h|^p\,\mathbb{E}[|\psi_\tau^{(N)}(e_{k - h})|^p]}{|h|^{s - 1}} \, d\tau \\
& \leq \sup_{k \in \mathbb{Z}^2_0} C_{p, \epsilon} \sup_{t \in [0, T]} \int_0^t \sum_{h \neq 0, k} \frac{|h|^p \,\mathbb{E}[|\Pi_N\psi_\tau(e_h)|^2]^{\frac{p}{2}}}{|k - h|^{s - 1}} + \sum_{h \neq 0, k} \frac{|k - h|^p\,\mathbb{E}[|\psi_\tau^{(N)}(e_{k - h})|^2]^{\frac{p}{2}}}{|h|^{s - 1}} \, d\tau \\
& \leq \sup_{k \in \mathbb{Z}^2_0} C_{p, \epsilon} \sup_{t \in [0, T]} \int_0^t \sum_{h \neq 0, k} \frac{1}{|k - h|^{s - 1}} + \sum_{h \neq 0, k} \frac{1}{|h|^{s - 1}} \, d\tau \\
& \leq C_{p, \epsilon}T,
\end{align*}
where since we have chosen $s > 3$ the above sums in $h$ converge, provided $\epsilon$ is chosen sufficiently small.

Next, we have from a Borel-Cantelli Lemma argument using the above estimate on $\mathcal{I}_N(T)$ that $\mathcal{I}_N(T) \to 0$ as $T \to 0$ almost surely, with the rate of convergence depending on $p$ and $\epsilon$.  Accordingly we may choose a time $\mathfrak{t} > 0$ for which $C_{p, \epsilon}\mathcal{I}_N(\mathfrak{t}) \leq \frac12$.  The differential inequality \eqref{DiffInequality} then implies that $\mathcal{D}_N \to 0$ as $N \to \infty$ almost surely.  This establishes the existence of a time $T > 0$ depending on $p, \epsilon, \psi_0$ on which pathwise uniqueness holds.

\textbf{Step 2 : Pathwise uniqueness on time intervals of arbitrary length.}  In order to extend this result to an arbitrary time interval $[0, T]$ we use a continuity argument.  Fixing an arbitrary interval $[0, T]$ and initial data $\psi_0$, and denote by $\Psi_T$ the ensemble of solutions $\psi(\omega, t)$ in the class $C([0, T] : \mathcal{F}L^{\infty, s})$ for which $\psi(\omega, 0) = \psi_0$.  Since $\Psi_T$ is nonempty, fix some particular solution $\psi(t)$, and denote the other solutions in $\Psi_T$ by $\psi_\omega(t)$.  Suppose that pathwise uniqueness does not hold: then there exists a first time $T_* < T$ for which $$\mathbb{P}\left(\sup_{[0, T_*]} \|\psi_\omega - \psi\|_{\mathcal{F}L^{\infty, s}} > 0 \right) = 0,$$ but for any $\mathcal{T} \in (T_*, T]$ we have $$\mathbb{P}\left(\sup_{[0, \mathcal{T}]} \|\psi_\omega - \psi\|_{\mathcal{F}L^{\infty, s}} > 0 \right) > 0.$$  However, if we now apply the uniqueness result proved in Step 1 at time $t = T_*$, we see that there is a nonvacuous interval $[T_*, T_* + \mathfrak{t}]$ with $\mathfrak{t} > 0$ depending on $p, \epsilon$ and $\psi(T_*)$ for which $$\mathbb{P}\left(\sup_{[0, T_* + \mathfrak{t}]} \|\psi_\omega - \psi\|_{\mathcal{F}L^{\infty, s}} > 0 \right) = 0,$$ which contradicts the assumed maximality of $T_*$.  We conclude that pathwise uniqueness holds on an interval $[0, T]$ of arbitrary length.

Since we have shown pathwise uniqueness in the class $C([0, T] : \mathcal{F}L^{\infty, 2\delta - 3\delta\epsilon})$, it holds \textit{a fortiori} in the class $C([0, T] : \mathcal{F}L^{\infty, (-\epsilon) \wedge ( 2\delta - 3\delta\epsilon)})$.  Finally, by applying this result successively to a sequence of times $T \to \infty$, we conclude that such pathwise unique solutions can be extended to the class $C([0, \infty) : \mathcal{F}L^{\infty, (-\epsilon) \wedge ( 2\delta - 3\delta\epsilon)})$.

\end{document}